\documentclass[12pt, eqno, twoside]{article}
\usepackage{graphicx}
\usepackage{amscd}
\usepackage[matrix,arrow,curve]{xy}
\usepackage{amssymb,amsmath}
\usepackage{amsmath,amsfonts,amssymb,latexsym}
\usepackage{fancyhdr}
\newtheorem{example}{Example}[section]
\newtheorem{theorem}{Theorem}[section]
\newtheorem{definition}{Definition}[section]
\newtheorem{proposition}{Proposition}[section]
\newtheorem{corollary}{Corollary}[section]
\newtheorem{lemma}{Lemma}[section]

\newenvironment{proof}[1][Proof]{\noindent\textbf{#1.} }{\
\rule{0.5em}{0.5em}}

\begin{document}

\pagestyle{fancy}
\fancyhead{} % clear all header fields
\fancyhead[EC]{\small\it PP Ntumba}%
\fancyhead[EL,OR]{\thepage} \fancyhead[OC]{\small\it On $\mathcal
A$-Transvections and Symplectic $\mathcal A$-Modules}%
\fancyfoot{} % clear all footer fields
\renewcommand\headrulewidth{0.5pt}
\addtolength{\headheight}{2pt} % make space for the rule

\title{\Large{\textbf{On $\mathcal A$-Transvections and
Symplectic $\mathcal A$-Modules}}}
\author{Patrice P. Ntumba}

\date{}
\maketitle

\begin{abstract}
In this paper, building on prior joint work by Mallios and Ntumba,
we show that $\mathcal A$-\textit{transvections} and
\textit{singular symplectic }$\mathcal A$-\textit{automorphisms} of
symplectic $\mathcal A$-modules of finite rank have properties
similar to the ones enjoyed by their classical counterparts. The
characterization of singular symplectic $\mathcal A$-automorphisms
of symplectic $\mathcal A$-modules of finite rank is grounded on a
newly introduced class of pairings of $\mathcal A$-modules: the
\textit{orthogonally convenient pairings.} We also show that, given
a symplectic $\mathcal A$-module $\mathcal E$ of finite rank, with
$\mathcal A$ a \textit{PID-algebra sheaf}, any injective $\mathcal
A$-morphism of a \textit{Lagrangian sub-$\mathcal A$-module}
$\mathcal F$ of $\mathcal E$ into $\mathcal E$ may be extended to an
$\mathcal A$-symplectomorphism of $\mathcal E$ such that its
restriction on $\mathcal F$ equals the identity of $\mathcal F$.
This result also holds in the more general case whereby the
underlying free $\mathcal A$-module $\mathcal E$ is equipped with
two symplectic $\mathcal A$-structures $\omega_0$ and $\omega_1$,
but with $\mathcal F$ being Lagrangian with respect to both
$\omega_0$ and $\omega_1$. The latter is the analog of the classical
\textit{Witt's theorem} for symplectic $\mathcal A$-modules of
finite rank.
\end{abstract}
%{\it Subject Classification (2000)}: 18F20, 14F05, 54B40.\\

\maketitle

\section*{Introduction}

Here is a further attempt of investigating classical notions/results
such as \textit{transvections}, \textit{Witt's theorem for
symplectic vector spaces}, the \textit{characterization of singular
symplectic automorphisms of symplectic vector spaces of finite
$($even$)$ dimension} within the context of Abstract Differential
Geometry (\`{a} la Mallios), \cite{mallios, modern}. This endeavor,
as already signalled in \cite{darboux}, is for the purpose of
rewriting and/or recapturing a great deal of classical symplectic
(differential) geometry without any use (\textit{at all}!) of any
notion of ``\textit{differentiability}" (differentiability is here
understood in the sense of the standard \textit{differential
geometry} of $C^\infty$-manifolds).

Now, we take the opportunity to review succinctly the basic notions
of Abstract Geometric Algebra which we are concerned with in this
paper. Most of the concepts in this paper are defined on the basis
of the classical ones; see to this effect, Artin~\cite{artin},
Berndt~\cite{berndt}, Crumeyrolle~\cite{crumeyrolle},
Deheuvels~\cite{deheuvels}, Lang~\cite{lang}. Our main reference is
Mallios~\cite{mallios}.

A \textit{$\mathbb C$-algebraized space} on a topological space $X$
is a pair $(X, \mathcal{A})$, where $\mathcal{A}\equiv (\mathcal{A},
\tau, X)$ is a (preferably unital and commutative) sheaf of $\mathbb
C$-algebras (or in other words, a $\mathbb{C}$-algebra sheaf). A
sheaf of sets $\mathcal{E}\equiv (\mathcal{E}, \rho, X)$ on $X$ is a
\textit{sheaf of groups} (or a \textit{group sheaf}) on $X$,
provided the following conditions are satisfied: (i) Each stalk of
$\mathcal E$ is a group; (ii) Given the set $\mathcal{E}\circ
\mathcal{E}:= \{(z, z')\in \mathcal{E}\times \mathcal{E}:\ \rho(z)=
\rho(z')\},$ the map $\mathcal{E}\circ \mathcal{E}\longrightarrow
\mathcal{E}: (z, z')\longmapsto z+ z'\in \mathcal{E}_x\subseteq
\mathcal E$ is continuous ($\mathcal{E}\circ \mathcal{E}$ is
equipped with the topology induced from $\mathcal{E}\times
\mathcal{E}$); A \textit{sheaf of $\mathcal A$-modules} (or an
$\mathcal A$-\textit{module}) on $X$ is a sheaf $\mathcal{E}\equiv
(\mathcal{E}, \rho, X)$ such that the following conditions hold: (i)
$\mathcal E$ is a sheaf of abelian groups; (ii) For every point
$x\in X$,  the corresponding stalk $\mathcal{E}_x$ of $\mathcal E$
is a (left) $\mathcal{A}_x$-module; (iii) The exterior module
multiplication in $\mathcal E$, viz. the map $\mathcal{A}\circ
\mathcal{E}\longrightarrow \mathcal{E}: (a, z)\longmapsto a\cdot
z\in \mathcal{E}_x\subseteq \mathcal E$ with $\tau(a)= \rho(z)= x\in
X$, is continuous. An $\mathcal A$-module $\mathcal E$ is called a
\textit{free $\mathcal A$-module of rank $n$} ($n\in \mathbb{N}$),
provided $\mathcal{E}= \mathcal{A}^n$ within an $\mathcal
A$-isomorphism. The $\mathcal A$-module $\mathcal{A}^n$ is called
the standard free $\mathcal A$-module of rank $n$. For an open
subset $U\subseteq X$, the canonical basis of the
$\mathcal{A}(U)$-module $\mathcal{A}^n(U)$ is the set
$\{\varepsilon^U_i\}_{1\leq i\leq n}$, where $\varepsilon_i^U:=
\delta^U_{ij}\in \mathcal{A}^n(U)\cong \mathcal{A}(U)^n$ such that
$\delta^U_{ij}=1$ for $i=j$ and $\delta^U_{ij}=0$ for $i\neq j$. So
one gets, for any $x\in X$, $\varepsilon^U_i(x)=
(\delta_{ij}^U(x))_{1\leq j\leq n}\in \mathcal{A}_x^n$ ($1\leq i\leq
n$), where $\delta^U_{ij}(x)= 1_x\in \mathcal{A}_x$, if $i= j$, and
$\delta^U_{ij}(x)= 0_x\in \mathcal{A}_x$, if $i\neq j$.

Now suppose there is given a presheaf of unital and commutative
$\mathbb C$-algebras $A\equiv (A(U), \tau^U_V)$ and a presheaf of
abelian groups $E\equiv (E(U), \rho^U_V)$, both on a topological
space $X$ and such that (i) $E(U)$ is a (left) $A(U)$-module, for
every open set $U\subseteq X$, (ii) For any open sets $U, V$ in $X$,
with $V\subseteq U$, $\rho^U_V(a\cdot s)= \tau^U_V(a)\cdot
\rho^U_V(s)$ for any $a\in A(U)$ and $s\in E(U)$. We call such a
presheaf $E$ a \textit{presheaf of $A(U)$-modules} on $X$, or simply
an \textit{$A$-presheaf} on $X$.

All our $\mathcal A$-modules and $A$-presheaves in this paper are
defined on a fixed topological space $X$. $\mathcal A$-modules and
$A$-presheaves with their respective morphisms form categories which
we denote $\mathcal{A}$-$\mathcal{M}od_X$ and
$A$-$\mathcal{P}\mathcal{S}h_X$ respectively. By virtue of the
equivalence $\mathcal{S}h_X\cong
\mathcal{C}o\mathcal{P}\mathcal{S}h_X$, an $\mathcal A$-morphism
$\phi: \mathcal{E}\longrightarrow \mathcal F$ of $\mathcal
A$-modules $\mathcal E$ and $\mathcal F$ may be identified with the
$A$-morphism $\overline{\phi}:= (\overline{\phi}_U)_{X\supseteq U,\
open}: E\longrightarrow F$ of the associated $A$-presheaves. We
shall most often denote by just $\phi$ the corresponding
$A$-morphism associated with the $\mathcal A$-morphism $\phi$. The
meaning of $\phi$ will always be determined by the situation at
hand.

Recall that given an $\mathcal A$-module $\mathcal E$ and a
sub-$\mathcal A$-module $\mathcal F$ of $\mathcal E$, the quotient
$\mathcal A$-module of $\mathcal E$ by $\mathcal F$ is the $\mathcal
A$-module generated by the presheaf sending an open $U\subseteq X$
to an $\mathcal{A}(U)$-module $S(U):= \Gamma(U,
\mathcal{E})/\Gamma(U, \mathcal{F})\equiv
\mathcal{E}(U)/\mathcal{F}(U)$ such that for every restriction map
$\sigma^U_V: \mathcal{E}(U)/\mathcal{F}(U)\longrightarrow
\mathcal{E}(V)/\mathcal{F}(V)$ one has $\sigma^U_V(r+
\mathcal{F}(U)):= \rho^U_V(r)+ \mathcal{F}(V)$ (the $\rho^U_V$ are
the restriction maps for the $A$-presheaf $\Gamma\mathcal{E}$).

For the sake of easy referencing, we also recall some notions, which
may be found in our recent papers such as \cite{malliosntumba3},
\cite{malliosntumba2}, \cite{darboux}, and \cite{orthogonally}. Let
$\mathcal F$ and $\mathcal E$ be $\mathcal A$-modules and $\phi:
\mathcal{F}\oplus \mathcal{E}\longrightarrow \mathcal{A}$ an
$\mathcal A$-bilinear morphism. Then, we say that the triplet
$[\mathcal{F}, \mathcal{E}; \mathcal{A}]\equiv [(\mathcal{F},
\mathcal{E}; \phi);\mathcal{A}]$ forms a \textit{pairing of
$\mathcal{A}$-modules} or that $\mathcal F$ and $\mathcal E$ are
\textit{paired through $\phi$ into $\mathcal A$.} The sub-$\mathcal
A$-module $\mathcal{F}^\perp$ of $\mathcal E$ such that, for every
open subset $U$ of $X$, $\mathcal{F}^\perp(U)$ consists of all $r\in
\mathcal{E}(U)$ with $\phi_V(\mathcal{F}(V), r|_V)=0$ for any open
$V\subseteq U$, is called the \textit{right kernel} of the pairing
$[\mathcal{F}, \mathcal{E}; \mathcal{A}]$. In a similar way, one
defines the \textit{left kernel} of $[\mathcal{F}, \mathcal{E};
\mathcal{A}]$ to be the sub-$\mathcal A$-module $\mathcal{E}^\perp$
of $\mathcal F$ such that, for any open subset $U$ of $X$,
$\mathcal{E}^\perp(U)$ is the set of all (local) sections $r\in
\mathcal{F}(U)$ such that $\phi_V(r|_V, \mathcal{E}(V))=0$ for every
open $V\subseteq U$. If $[(\mathcal{F}, \mathcal{E}; \phi);
\mathcal{A}]$ is a pairing of free $\mathcal A$-modules, then, for
every open subset $U$ of $X$, $\mathcal{F}^\perp(U)=
\mathcal{F}(U)^\perp:= \{r\in \mathcal{E}(U):\
\phi_U(\mathcal{F}(U), r)=0\}$, and similarly $\mathcal{E}^\perp(U)=
\mathcal{E}(U)^\perp:= \{r\in \mathcal{F}(U): \phi_U(r,
\mathcal{E}(U))=0\}$.

Now, let $[(\mathcal{E}, \mathcal{E}; \phi); \mathcal{A}]$ be a
pairing such that if $r, s\in \mathcal{E}(U)$, where $U$ is an open
subset of $X$, then $\phi_U(r, s)=0$ if and only if $\phi_U(s,
r)=0.$ The left kernel, $\mathcal{E}^\perp_l:= \mathcal{E}^\perp$,
is the same as the right kernel $\mathcal{E}^\perp_r:=
\mathcal{E}^\top.$ In that case, we say that the $\mathcal
A$-bilinear form $\phi$ is \textit{orthosymmetric} and call
$\mathcal{E}^\perp(=\mathcal{E}^\top)$ the \textit{radical sheaf}
(or \textit{sheaf of $\mathcal{A}$-radicals,} or simply
\textit{$\mathcal A$-radical}) of $\mathcal E$, and denote it by
rad$_\mathcal{A}\mathcal{E}\equiv \mbox{rad}\ \mathcal{E}$. An
$\mathcal A$-module $\mathcal E$ such that rad $\mathcal{E}\neq 0$
(resp. rad $\mathcal{E}=0$) is called \textit{isotropic} (resp.
\textit{non-isotropic}); $\mathcal E$ is \textit{totally isotropic}
if $\phi$ is identically zero, i.e. $\phi_U(r, s)=0$ for all
sections $r, s\in \mathcal{E}(U)$, with $U$ open in $X$. For any
open $U\subseteq X$, a non-zero section $r\in \mathcal{E}(U)$ is
called \textit{isotropic} if $\phi_U(r, r)=0.$ The $\mathcal
A$-radical of a sub-$\mathcal A$-module $\mathcal F$ of $\mathcal E$
is defined as rad $\mathcal{F}:= \mathcal{F}\cap \mathcal{F}^\perp=
\mathcal{F}\cap \mathcal{F}^\top.$ If $[\mathcal{F}, \mathcal{E};
\mathcal{A}]$ is a pairing of free $\mathcal A$-modules, then for
every open subset $U$ of $X$, $(\mbox{rad}\ \mathcal{E})(U)=
\mbox{rad}\ \mathcal{E}(U)$ and $(\mbox{rad}\ \mathcal{F})(U)=
\mbox{rad}\ \mathcal{F}(U),$ where $\mbox{rad}\ \mathcal{E}(U)=
\mathcal{E}(U)\cap \mathcal{E}(U)^\perp$ and $\mbox{rad}\
\mathcal{F}(U)= \mathcal{F}(U)\cap \mathcal{F}(U)^\perp$. Given a
pairing $[(\mathcal{E}, \mathcal{E}; \phi); \mathcal{A}]$ with
$\phi$ a \textit{symmetric} $\mathcal A$-bilinear morphism,
sub-$\mathcal A$-modules $\mathcal{E}_1$ and $\mathcal{E}_2$ of
$\mathcal E$ are said to be \textit{mutually orthogonal} if for
every open subset $U$ of $X$, $\phi_U(r, s)=0,$ for all $r\in
\mathcal{E}_1(U)$ and $s\in\mathcal{E}_2(U).$ If $\mathcal{E}=
\oplus_{i\in I}\mathcal{E}_i$, where the $\mathcal{E}_i$ are
pairwise orthogonal sub-$\mathcal{A}$-modules of $\mathcal E$, we
say that $\mathcal E$ is the direct orthogonal sum of the
$\mathcal{E}_i$, and write $\mathcal{E}:= \mathcal{E}_1\bot\cdots
\bot \mathcal{E}_i\bot\cdots .$

\textbf{N.B.} We assume throughout the paper, unless otherwise
mentioned, that the pair $(X, \mathcal{A})$ is an
\textit{algebraized space}, where $\mathcal{A}$ is a \textit{unital
$\mathbb{C}$-algebra sheaf} such that \textit{every nowhere-zero
section of $\mathcal{A}$} is \textit{invertible.} Furthermore,
\textit{all free $\mathcal A$-modules} are considered to be
\textit{torsion-free,} that is, for any open subset $U\subseteq X$
and \textit{nowhere-zero} section $s\in \mathcal{E}(U)$, if $as=0$,
where $a\in \mathcal{A}(U)$, then $a=0$.

\section{$\mathcal A$-transvections and useful background}

For the purpose of this section, we recall the following useful
facts/results, which are found in our papers \cite{malliosntumba3},
\cite{noteorthogonally} and \cite{secondiso}. For these results,
\textit{$\mathcal A$ is just an arbitrary unital sheaf of
$\mathbb{C}$-algebras with no other condition on.}

\begin{theorem}\label{theorem2}
Let $\mathcal E$ be an $\mathcal A$-module and $\mathcal F$ and
$\mathcal G$ sub-$\mathcal A$-modules of $\mathcal E$. Then,
\[
\mathcal{G}/(\mathcal{F}\cap \mathcal{G})\cong
\mathbf{S}(\Gamma(\mathcal{F})+ \Gamma(\mathcal{G}))/\mathcal{F}.
\]
\end{theorem}
Theorem \ref{theorem2} yields the following result.

\begin{corollary}\label{corollary1}
Let $\mathcal E$ be an $\mathcal A$-module, $\mathcal F$ and
$\mathcal G$ sub-$\mathcal A$-modules of $\mathcal E$ such that
$\mathcal{E}\cong \mathcal{F}\oplus \mathcal{G}.$ Then,
\[
\mathcal{E}/\mathcal{F}\cong \mathcal{G}.
\]
\end{corollary}

\begin{corollary}\label{corollary2}
If $\mathcal F$ is a free sub-$\mathcal A$-module of a free
$\mathcal A$-module $\mathcal E$, then the quotient
$\mathcal{E}/\mathcal{F}$ is free and for every open $U\subseteq X$,
\[(\mathcal{E}/\mathcal{F})(U)= \mathcal{E}(U)/\mathcal{F}(U)
\]
within an $\mathcal{A}(U)$-isomorphism. Moreover,
\[\mbox{rank}\ \mathcal{F}+ \mbox{corank}\ \mathcal{F}= \mbox{rank}\
\mathcal{E},
\]
where corank $\mathcal{F}:= \mbox{rank}\ (\mathcal{E}/\mathcal{F}).$
\end{corollary}

Let us recall the following definition.

\begin{definition}
\emph{Let $\mathcal E$ be a free $\mathcal A$-module and $\mathcal
F$ a free sub-$\mathcal A$-module of $\mathcal E$, complemented in
$\mathcal E$ be a free sub-$\mathcal A$-module $\mathcal G$. The
rank of $\mathcal{G}\cong \mathcal{E}/\mathcal{F}$ is called the
\textbf{corank} of $\mathcal F$, that is,
\[
\mbox{corank}\ \mathcal{F}= \mbox{rank}(\mathcal{E}/\mathcal{F}).
\]}
\end{definition}

Furthermore, let us also state the following results.

\begin{theorem}
Let $\mathcal E$ be a free $\mathcal A$-module, and $\mathcal F$ and
$\mathcal G$ free sub-$\mathcal A$-modules of $\mathcal E$ such that
$\mathcal{F}\cap \mathcal{G}$ and $\mathcal{F}+ \mathcal{G}$ are
free. Then, \[\begin{array}{rll} \mbox{rank}\ \mathcal{F}+
\mbox{corank}\ \mathcal{F} & = & \mbox{rank}\ \mathcal{E}\\
\mbox{rank}\ (\mathcal{F}+ \mathcal{G})+ \mbox{rank}\
(\mathcal{F}\cap \mathcal{G}) & = & \mbox{rank}\ \mathcal{F}+
\mbox{rank} \mathcal{G}\\ \mbox{corank}\ (\mathcal{F}+ \mathcal{G})+
\mbox{corank}\ (\mathcal{F}\cap \mathcal{G}) & = & \mbox{corank}\
\mathcal{F} + \mbox{corank}\ \mathcal{G}. \end{array}\]
\end{theorem}

\begin{theorem}\label{theo2}
Let $\mathcal{E}$ be a free $\mathcal{A}$-module of arbitrary rank.
Then, for any open subset $U\subseteq X$, $\mbox{rank}\
\mathcal{E}^\ast(U)= \mbox{rank}\ \mathcal{E}(U).$ Fix an open set
$U$ in $X$. If $\psi\equiv (\psi_V)_{U\supseteq V,\ open}\in
\mathcal{E}^\ast(U)$ and $\psi_U(s)=0$ $($ which implies that
$\psi_V(s|_V)=0$ for any open $V\subseteq U$ $)$ for all $s\in
\mathcal{E}(U),$ then $\psi=0;$ on the other hand if $\psi_U(s)=0$
for all $\psi\in \mathcal{E}^\ast(U),$ then $s=0.$ Finally, let
$\mbox{rank}\ \mathcal{E}(U)= n$. To a given basis $\{s_i\}$ of
$\mathcal{E}(U)$, we can find a \textsf{dual basis} $\{\psi_i\}$ of
$\mathcal{E}^\ast(U)\cong \mathcal{E}(U),$ where \[\psi_{i,
V}(s_j|_V)= \delta_{ij, V}\in \mathcal{A}(V)\]for any open
$V\subseteq U.$
\end{theorem}

\begin{theorem}\label{theo3}
Let $\mathcal F$ and $\mathcal E$ be $\mathcal A$-modules paired
into a $\mathbb{C}$-algebra sheaf $\mathcal A$, and assume that
$\mathcal{E}^\perp=0.$ Moreover, let $\mathcal{F}_0$ be a
sub-$\mathcal A$-module of $\mathcal F$ and $\mathcal{E}_0$ a
sub-$\mathcal A$-module of $\mathcal E$. There exist natural
$\mathcal A$-isomorphisms into: \begin{equation}
\mathcal{E}/\mathcal{F}_0^\perp\longrightarrow \mathcal{F}_0^\ast,
\label{eq9}\end{equation} and
\begin{equation}\label{eq10}\mathcal{E}_0^\perp\longrightarrow
(\mathcal{E}/\mathcal{E}_0)^\ast.\end{equation}
\end{theorem}

In keeping with the notations of Corollary \ref{corollary2}, the
free sub-$\mathcal A$-module $\mathcal F$ is called an $\mathcal
A$-\textit{hyperplanee} if corank $\mathcal{F}=1$, i.e. the quotient
$\mathcal A$-module $\mathcal{E}/\mathcal{F}$ is a line $\mathcal
A$-module. On the other hand, we notice that if $\phi\equiv
(\phi_U)_{X\supseteq U,\ open}$ is an $\mathcal A$-endomorphism of
$\mathcal E$, then $\phi$ induces on the line $\mathcal A$-module
$\mathcal{E}/\mathcal{F}$ an $\mathcal A$-\textit{homothecy}, which
we denote by $\widetilde{\phi}$. More explicitly, if $U$ is open in
$X$ and $s$ a section of $\mathcal{E}/\mathcal{F}$ over $U$, then
\[\widetilde{\phi}(s)\equiv \widetilde{\phi}_U(s)= a_Us\equiv as
\]
for some $a_U\equiv a\in \mathcal{A}(U).$ The coefficient sections
$a_U$ are such that $a_V= {a_U}|_V$ whenever $V$ is contained in
$U$. The global section $a_X\equiv a$ is called the \textit{ratio}
of the $\mathcal A$-homothecy $\widetilde{\phi}.$

\begin{proposition}\label{propositio1}
Let $\mathcal E$ be a free $\mathcal A$-module, and $\mathcal F$ a
proper free sub-$\mathcal A$-module of $\mathcal E$. Then, the
following assertions are equivalent.
\begin {enumerate}
\item [${(1)}$] $\mathcal F$ is an $\mathcal A$-hyperplane of $\mathcal E$.
\item [${(2)}$] For every $($local$)$ section $s\in \mathcal{E}(U)$
such that $s|_V\notin \mathcal{F}(V)$ for every open $V\subseteq U$,
\[
\mathcal{E}(U)\cong \mathcal{A}(U)s\oplus \mathcal{F}(U).
\]
\item [${(3)}$] For every open $U\subseteq X$, there exists a
section $s\in \mathcal{E}(U)$ with $s|_V\notin \mathcal{F}(V),$
where $V$ is any open subset contained in $U$, such that
\[
\mathcal{E}(U)\cong \mathcal{A}(U)s\oplus \mathcal{F}(U).
\]
\item [${(4)}$] The free sub-$\mathcal A$-module $\mathcal F$ is a
maximal sub-$\mathcal A$-module in the inclusion-ordered set of
proper free sub-$\mathcal A$-modules of $\mathcal E$.
\end{enumerate}
\end{proposition}

\begin{proof}
$(1)\Rightarrow (2)$: For every open $U\subseteq X$ and section
$s\in \mathcal{E}(U)$ such that $s|_V\notin \mathcal{F}(V)$ for any
open $V\subseteq U$, it is clear that $\mathcal{A}(U)s+
\mathcal{F}(U)$ is a direct sum. On the other hand, the equivalence
class containing $s$ is a nowhere-zero section of
$\mathcal{E}/\mathcal{F}$; it spans $\mathcal{E}(U)/\mathcal{F}(U)$
since $\mathcal{E}(U)/\mathcal{F}(U)$ has rank $1$. It thus follows
that $\mathcal{E}(U)\cong \mathcal{A}(U)s+ \mathcal{F}(U).$

$(2)\Rightarrow (3)$: Evident.

$(3)\Rightarrow (1)$: Since
\[\mbox{rank} (\mathcal{E}/\mathcal{F})(U)= \mbox{rank}
\mathcal{E}(U)/\mathcal{F}(U)= \mbox{rank} \mathcal{A}(U)s=1
\]
for every open $U\subseteq X$ and $s\in \mathcal{E}(U)$ with
$s|_V\notin \mathcal{F}(V),$ where $V$ is any open subset contained
in $U$.

$(2)\Rightarrow (4)$: Let $\mathcal{F}'$ be a free sub-$\mathcal
A$-module of $\mathcal E$ containing $\mathcal F$ and such that rank
$\mathcal{F}'> \mbox{rank}\ \mathcal F$. For every open $U$ there
exists a section $s\mathcal{F}'(U)$ such that $s|_V\notin
\mathcal{F}(V)$ for every open $V\subseteq U$. By $(2)$, for every
open $U\subseteq X$, $\mathcal{E}(U)\cong \mathcal{A}(U)s\oplus
\mathcal{F}(U);$ but $\mathcal{A}(U)s\oplus \mathcal{F}(U)$ is
contained in $\mathcal{F}(U)$, therefore $\mathcal{F}'= \mathcal E$
within an $\mathcal A$-isomorphism.

$(4)\Rightarrow (2)$: Let $U$ be an open set in $X$. There exists a
section $s\in \mathcal{E}(U)$ with $s|_V\notin \mathcal{F}(V)$ for
any open $V\subseteq U;$ then $\mathcal{A}(U)s\oplus \mathcal{F}(U)$
contains strictly $\mathcal{F}(U),$ thus $\mathcal{A}(U)s\oplus
\mathcal{F}(U)\cong \mathcal{E}(U),$ since $\mathcal F$ is maximal.
\end{proof}

So we come to Theorem \ref{theorem1}, which characterizes the notion
of $\mathcal A$-\textit{transvection.} For the classical notion, see
\cite{artin}, \cite[p. 152, Proposition 12.9]{chambadal},
\cite{crumeyrolle}, \cite[p. 419 ff]{deheuvels}, \cite{dieudonne},
\cite[p. 542- 544]{lang}.

\begin{theorem} \label{theorem1}
Let $\mathcal E$ be a free $\mathcal A$-module, $\mathcal H$ an
$\mathcal A$-hyperplane of $\mathcal E$, $\phi\in
\mathcal{E}nd_\mathcal{A}\mathcal{E}$ such that $\phi(s)\equiv
\phi_U(s)=s$ for any $s\in \mathcal{H}(U),$ where $U$ is an
arbitrary open subset of $X$, and $\widetilde{\phi}$ the $\mathcal
A$-\textsf{homothecy induced} by $\phi$ on the line $\mathcal
A$-module $\mathcal{E}/\mathcal{H}$. Moreover, let $a\in
\mathcal{A}(X)$ be the \textsf{ratio of} $\widetilde{\phi}.$ Then,
$a$ is either zero or nowhere zero, and the following hold.
\begin{enumerate}
\item [${(1)}$] If $a|_U\equiv a_U\neq 1$ for every open $U\subseteq
X$, there exists a unique line $\mathcal A$-module
$\mathcal{L}\subseteq \mathcal{E}$ such that $\mathcal{E}\cong
\mathcal{H}\oplus \mathcal{L}$ and $\mathcal{L}$ is stable by
$\phi$, i.e. $\phi(\mathcal{L})\cong \mathcal{L}.$

\item[${(2})$] If $a=1,$ then for every $\mathcal A$-morphism
$\theta\in \mathcal{E}^\ast\equiv
\mathcal{H}om_\mathcal{A}(\mathcal{E}, \mathcal{A})$ with $\ker
\theta\cong \mathcal H$, there exists, for every open subset
$U\subseteq X$, a unique section $r\in \mathcal{H}(U)$ such that
\begin{equation}\label{1.2}
\phi(s)= s+ \theta(s)r
\end{equation}
for every $s\in \mathcal{E}(U).$
\end{enumerate}
\end{theorem}

\begin{proof}\ Clearly, as $\widetilde{\phi}\in
\mathcal{E}nd_\mathcal{A}(\mathcal{E}/\mathcal{H})$ and
$\mathcal{E}/\mathcal{H}$ is a line $\mathcal A$-module, and
\[
\mbox{rank}\ \mathcal{H}(U)+
\mbox{rank}(\mathcal{E}/\mathcal{H})(U)= \mbox{rank}\
\mathcal{E}(U),
\]
for every open $U\subseteq X$, it follows that $a$ is either zero or
nowhere zero.

\textbf{Assertion $(1).$} \underline{\textit{Uniqueness.}} Let
$\mathcal L$ be a line $\mathcal A$-module that complements
$\mathcal H$ in $\mathcal E$ and stable by the $\mathcal A$-morphism
$\phi$, and $s$ a nowhere-zero section of $\mathcal L$ on $X$ (such
a section $s$ does exist because $\mathcal{L}\cong \mathcal A$ and
$\mathcal A$ is unital). Therefore, there exists $b\in
\mathcal{A}(X)$ such that $\phi(s)= bs$. Next, assume that $q\equiv
(q_U)_{X\supseteq U,\ open}$ is the canonical $\mathcal A$-morphism
of $\mathcal E$ onto $\mathcal{E}/\mathcal{H}$. It is clear that
$\widetilde{\phi}_X(q_X(s))= bq_X(s)\equiv bq(s);$ thus
$\widetilde{\phi}_X$ is a homothecy of ratio $a=b$, hence, by
hypothesis, $b$ is nowhere $1$. Now, let $u$ be an element of
$\mathcal{E}(X)$ such that $u\notin \mathcal{H}(X);$ then there
exists a non-zero $\lambda\in \mathcal{A}(X)$ and an element $t\in
\mathcal{H}(X)$ such that
\[
u= \lambda s+ t.
\]
It follows that
\[
\phi(u)= \lambda b s+t.
\]
Of course, $\phi(u)$ and $u$ are colinear if and only if $t=0.$
Thus, we have proved that every section $u\in \mathcal{E}(X)$ which
is colinear with its image $\phi(u)$ belongs to $\mathcal{L}(X)$. A
similar argument holds should we consider the decomposition
$\mathcal{E}(U)\cong \mathcal{H}(U)\oplus \mathcal{L}(U)$, where $U$
is any other open subset $U$ of $X$. Hence, $\mathcal L$ is the
unique complement of $\mathcal H$ in $\mathcal E$, up to $\mathcal
A$-isomorphism, and stable by $\phi$.

\underline{\textit{Existence.}} Since $a$ is nowhere $1$ on $X$,
there exists a nowhere-zero section $s\in \mathcal{E}(X)$ such that
\[
\widetilde{\phi}_U(q_U(s|_U)):= \widetilde{\phi}_U(q_U(s_U))\neq
q_U(s_U)=: q_U(s|_U)
\]
for any open $U\subseteq X$. As $\widetilde{\phi}\circ q= q\circ
\phi,$ it follows that $r_U:= \phi_U(s_U)-s_U$ does not belong to
$\mathcal{H}(U),$ for any open $U\subseteq X.$ the line $\mathcal
A$-module $\mathcal{L}:= [r_U]_{X\supseteq U,\ open}$ clearly
complements $\mathcal H$. It remains to show that $\mathcal L$ is
stable by $\phi$: To this end, we first observe that every $s_U$
does not belong to the corresponding $\mathcal{H}(U),$ and, by
Proposition \ref{propositio1}, $\mathcal{E}(U)\cong
\mathcal{A}(U)s_U\oplus \mathcal{H}(U).$ So, since $r_U\notin
\mathcal{H}(U)$ for every open $U\subseteq X$, there exists for
every $r_U$ sections $\alpha_U\in \mathcal{A}(U)$ and $t_U\in
\mathcal{H}(U)$ such that
\begin{equation}\label{1.1}
r_U= \alpha_Us_U+ t_U.
\end{equation}
We deduce from (\ref{1.1}) that
\[
\phi_U(r_U)= (\alpha_U+1)r_U,
\]
and the proof is complete.

\textbf{Assertion $2$.} \underline{\textit{Uniqueness.}}\ Let us fix
an open set $U$ in $X$. The uniqueness of $r$ such that (\ref{1.2})
holds is immediate, as $\theta_U(s)\equiv \theta(s)\neq 0$ for some
$s\in \mathcal{E}(U).$ Relation (\ref{1.2}) also shows that if $s\in
\mathcal{E}(U)$ and $\theta(s)$ is nowhere zero, then necessarily
\[
r= (\theta(s))^{-1}(\phi(s)-s).
\]

\underline{\textit{Existence.}}\ Suppose given a section $s_0\in
\mathcal{E}(U)$ such that ${s_0}|_V\notin \mathcal{H}(V)$ for any
open $V\subseteq U$. Let us consider the section $r=
(\theta(s_0))^{-1}(\phi(s_0)- s_0).$ Clearly, $r\in \mathcal{H}(U);$
indeed
\[
(q\circ \phi)(s_0)- q(s_0)= (\widetilde{\phi}\circ q)(s_0)-
q(s_0)=0.
\]
The two $\mathcal{A}(U)$-morphisms $s\longmapsto \phi(s)$ and
$s\longmapsto s+ \theta(s)r$ are equal, since they take on, on one
hand, the same value at $s_0$, and, on the other hand, the same
value at every $s\in \mathcal{H}(U).$
\end{proof}

\begin{definition}
\emph{Let $\mathcal E$ be a free $\mathcal A$-module, $\mathcal H$
an $\mathcal A$-hyperplane of $\mathcal E$, $\phi\in
\mathcal{E}nd_\mathcal{A}\mathcal E$ such that $\phi(s)\equiv
\phi_U(s)= s$ for every $s\in \mathcal{H}(U)$, where $U$ is any open
subset of $X$, and that the $\mathcal A$-homothecy
$\widetilde{\phi}$, induced by $\phi$, has ratio $a\in
\mathcal{A}(X)$: $a=1$. Then, $\phi$ is called an $\mathcal
A$-\textbf{transvection} of $\mathcal E$, with respect to the
$\mathcal A$-hyperplane $\mathcal H$.}
\end{definition}

We shall now be led to \textit{adjoints} of $\mathcal A$-morphisms.
More precisely,

\begin{definition}\label{definition1}
\emph{Let $\mathcal E$ and $\mathcal F$ be $\mathcal A$-modules and
$\theta$ an $\mathcal A$-morphism of $\mathcal E$ into $\mathcal F$.
The $\mathcal A$-morphism ${}^t\theta\in
\mbox{H}om_\mathcal{A}(\mathcal{F}^\ast, \mathcal{E}^\ast)$ such
that, for every open set $U\subseteq X$ and sections $\psi\in
\mathcal{F}^\ast(U),$ $s\in \mathcal{E}(V),$ where $V\subseteq U$ is
open,
\begin{equation}\label{1.6}
[({}^t\theta)_U(\psi)](s):= \psi_V(\theta_V(s))
\end{equation}
is called the \textbf{transpose} of the $\mathcal A$-morphism
$\theta.$}
\end{definition}
It is clear from Definition \ref{definition1} that \textit{every
$\mathcal A$-morphism $\theta\in \mbox{H}om_\mathcal{A}(\mathcal{E},
\mathcal{F})$ admits a unique transpose} ${}^t\theta\in
\mbox{H}om_\mathcal{A}(\mathcal{F}^\ast, \mathcal{E}^\ast).$ Note
also that, in general, for every open subset $U\subseteq X$,
\begin{equation}\label{1.7}
({}^t\theta)_U\neq {}^t(\theta_U).
\end{equation}
Indeed, $({}^t\theta)_U$ is the $\mathcal{A}|_U$-map
$\mathcal{F}^\ast(U)\longrightarrow \mathcal{E}^\ast(U)$, given by
the formula (\ref{1.6}) above, whereas ${}^t(\theta_U)$ is the
$\mathcal{A}(U)$-map $(\mathcal{F}(U))^\ast\longrightarrow
(\mathcal{E}(U))^\ast$ such that if $\psi\in (\mathcal{F}(U))^\ast$,
i.e. an $\mathcal{A}(U)$-linear map on $\mathcal{F}(U)$, and $s\in
\mathcal{E}(U),$ then
\[
{}^t(\theta_U)(\psi)(s):= \psi(\theta_U(s)).
\]
The inequality (\ref{1.7}) means that \textit{transposes do not
commute with restrictions.}

\begin{definition}
\emph{Let $\mathcal E$ be an $\mathcal A$-module and $\phi$ an $\mathcal
A$-bilinear form on $\mathcal E$. The $\mathcal A$-bilinear form
$\phi^\ast$ on $\mathcal E$ such that, for any open set $U$ in $X$
and sections $s, t\in \mathcal{E}(U),$
\[
\phi^\ast(s, t)\equiv \phi^\ast_U(s, t)= \phi_U(t, s)\equiv \phi(t,
s)
\]
is called the \textbf{adjoint} of $\phi$.}
\end{definition}

Clearly, $\phi^\ast= \phi$ if and only if $\phi$ is symmetric; for
$\phi^\ast= -\phi$ it is necessary and sufficient that $\phi$ be
antisymmetric. An $\mathcal A$-bilinear form $\phi$ is called
\textbf{self-adjoint} (resp. \textbf{skew-adjoint}) if $\phi^\ast=
\phi$ (resp. $\phi^\ast= -\phi$).

In classical multilinear algebra (see e.g. \cite[p. 339, Definition
20.1]{chambadal}, \cite[pp, 144, 145]{lang}), one may associate with
a given bilinear form two linear maps: the \textit{right insertion
map} and the \textit{left insertion map.} The corresponding
situation for $\mathcal A$-bilinear forms is as follows.

\begin{definition}\label{definition2}
\emph{Let $\mathcal E$ and $\mathcal F$ be $\mathcal A$-modules, and
$\phi: \mathcal{E}\oplus \mathcal{F}\longrightarrow \mathcal A $ an
$\mathcal A$-bilinear form. The $\mathcal A$-morphism
\begin{equation}\label{1.3}
\phi^R\in \mbox{H}om_\mathcal{A}(\mathcal{F},
\mathcal{E}^\ast)\equiv \mbox{H}om_\mathcal{A}(\mathcal{F},
\mathcal{H}om_\mathcal{A}(\mathcal{E}, \mathcal{A}))
\end{equation}
such that, for any open subset $U\subseteq X$ and sections $t\in
\mathcal{F}(U)$ and $s\in \mathcal{E}(V)$, where $V\subseteq U$ is
open,
\[
\phi^R_U(t)(s)\equiv (\phi^R)_U(t)(s):= \phi_V(s, t|_V)
\]
is called the \textbf{right insertion $\mathcal A$-morphism}
associated with the $\mathcal A$-bilinear form $\phi.$ Similarly,
for every open subset $U\subseteq X$ and sections $s\in
\mathcal{E}(U)$ and $t\in \mathcal{F}(V)$, where $V\subseteq U$ is
open,
\[
\phi^L_U(s)(t)\equiv (\phi^L)_U(s)(t):= \phi_V(s|_V, t)
\]
defines an $\mathcal A$-morphism $\phi^L$ of $\mathcal E$ into
$\mathcal{F}^\ast$, i.e.
\[
\phi^L\in \mbox{H}om_\mathcal{A}(\mathcal{E},
\mathcal{F}^\ast)\equiv \mbox{H}om_\mathcal{A}(\mathcal{E},
\mathcal{H}om_\mathcal{A}(\mathcal{F}, \mathcal{A})).
\]
The $\mathcal A$-morphism $\phi^L$ is called the \textbf{left
insertion $\mathcal A$-morphism} associated with $\phi$.}
\end{definition}

It is clear in the light of Definition \ref{definition2} that if the
$\mathcal A$-\textit{bilinear form} $\phi: \mathcal{E}\oplus
\mathcal{F}\longrightarrow \mathcal A$ \textit{is non-degenerate,}
then \textit{both insertion $\mathcal A$-morphisms $\phi^R$ and
$\phi^L$ are injective.}

\begin{definition}\label{definition3}
\emph{Let $\mathcal E$ and $\mathcal{E}'$ be $\mathcal A$-modules,
$\phi$ and $\phi'$ non-degenerate $\mathcal A$-bilinear forms on
$\mathcal E$ and $\mathcal{E}'$, respectively. Moreover, let $\psi$
be an $\mathcal A$-morphism of $\mathcal E$ into $\mathcal{E}'$. An
$\mathcal{A}$-morphism $\theta\in
\mbox{H}om_\mathcal{A}(\mathcal{E}', \mathcal{E})$ such that, for
every open subset $U\subseteq X$ and sections $s\in \mathcal{E}(U),$
$t\in \mathcal{E}'(U),$
\[
\phi'(\psi(s), t)\equiv \phi'_U(\psi_U(s), t)= \phi_U(s,
\theta_U(t))\equiv \phi(s, \theta(t))
\]
is called an \textbf{adjoint} of $\psi,$ and is denoted
$\psi^\ast.$}
\end{definition}

Keeping with the notations of Definition \ref{definition3} above, we
have

\begin{proposition}
$\theta$ is unique whenever it exists.
\end{proposition}

\begin{proof}
Let $\theta_1$ and $\theta_2$ be $\mathcal A$-morphisms of
$\mathcal{E}'$ into $\mathcal E$ such that, given any open subset
$U\subseteq X$ and sections $s\in \mathcal{E}(U),$ $t\in
\mathcal{E}'(U),$
\begin{equation}\label{1.4}
\phi_U'(\psi_U(s), t)= \phi_U(s, \theta_{1, U}(t))= \phi_U(s,
\theta_{2, U}(t)).
\end{equation}
Using the right insertion $\mathcal A$-morphism $\phi^R$, it follows
form (\ref{1.4}) that
\[
\phi^R_U(\theta_{1, U}(t))(s)= \phi^R_U(\theta_{2, U}(t))(s).
\]
Since $s$ is arbitrary in $\mathcal{E}(U)$,
\[
\phi^R_U(\theta_{1, U}(t))= \phi^R_U(\theta_{2, U}(t)).
\]
But $\phi^R$ is injective, therefore
\[
\theta_{1, U}= \theta_{2, U}.
\]
Finally, since $U$ is arbitrary, $\theta_1= \theta_2.$
\end{proof}

Let us now enquire the existence of the adjoint of an $\mathcal
A$-morphism $\psi\in \mbox{H}om_\mathcal{A}(\mathcal{E},
\mathcal{E}'),$ where $\mathcal E$ and $\mathcal{E}'$ are $\mathcal
A$-modules equipped with $\mathcal A$-bilinear forms $\phi$ and
$\phi'$, respectively.

\begin{proposition}
Let $\mathcal E$ and $\mathcal{E}'$ be $\mathcal{A}$-modules,
equipped with non-degenerate $\mathcal A$-bilinear forms $\phi$ and
$\phi'$, respectively. If $\mathcal E$ is \textsf{free} and
\textsf{of finite rank}, then for every $\mathcal A$-morphism
$\psi\in \mbox{H}om_\mathcal{A}(\mathcal{E}, \mathcal{E}')$ there
exists an adjoint, denoted $\psi^\ast$, which is given by
\[
\psi^\ast= (\phi^R)^{-1}\circ {}^t\psi\circ {\phi'}^R,
\]
where ${}^t\psi: (\mathcal{E}')^\ast\longrightarrow
\mathcal{E}^\ast$ is the transpose of $\psi.$
\end{proposition}

\begin{proof}
Let $U$ be an open subset of $X$, $s\in \mathcal{E}(U)$ and $t\in
\mathcal{E}'(U)$. Using the right insertion $\mathcal{A}$-morphism
${\phi'}^R,$ one has
\begin{equation}\label{1.5}
\phi_U'(\psi_U(s), t)= {\phi'}^R_U(t)(\psi_U(s))=
({}^t\psi)_U({\phi'}^R_U(t))(s).
\end{equation}
Since $\mathcal E$ has finite rank and $\phi$ is non-degenerate,
$\phi^R$ is an $\mathcal A$-isomorphism of $\mathcal E$ onto
$\mathcal{E}^\ast$; so ${}^t\psi\circ {\phi'}^R$ may be written
\[
{}^t\psi\circ {\phi'}^R =\phi^R\circ ((\phi^R)^{-1}\circ
{}^t\psi\circ {\phi'}^R).
\]
It follows from (\ref{1.5}) that
\[\begin{array}{lll}
\phi_U'(\psi_U(s), t) & = & [\phi^R_U(((\phi_U^R)^{-1}\circ
({}^t\psi)_U\circ {\phi'}_U^R)(t))](s)\\ & = &  \phi_U( s,
((\phi_U^R)^{-1}\circ ({}^t\psi)_U\circ {\phi'}_U^R)(t)),
\end{array}
\]
which ends the proof.
\end{proof}

\begin{corollary}
Adjoints commute with restrictions.
\end{corollary}

\begin{proof}
Let $\mathcal E$ and $\mathcal{E}'$ be $\mathcal A$-modules, $\phi$
and $\phi'$ non-degenerate $\mathcal A$-bilinear forms on $\mathcal
E$ and $\mathcal{E}'$, respectively. Assume that $\psi\in
\mbox{H}om_\mathcal{A}(\mathcal{E}, \mathcal{E}').$ Let $U$ be an
open subset of $X$, and $s,\ t$ be sections of $\mathcal E$ and
$\mathcal{E}'$ on $U$, respectively. By Definition
\ref{definition3}, we have
\[
\phi_U'(\psi_U(s), t)= \phi_U(s, (\psi^\ast)_U(t)).
\]
On the other hand, since $\phi_U$ and $\phi'_U$ are non-degenerate
and
\[\psi_U\in Hom_{\mathcal{A}(U)}(\mathcal{E}(U),
\mathcal{E}'(U)),
\]
then by virtue of \cite[pp. 385,
386]{chambadal}, we have
\[
\phi'_U(\psi_U(s), t)= \phi_U(s, (\psi_U)^\ast(t)).
\]
On account of uniqueness of adjoints, we have
\[
(\psi^\ast)_U= (\psi_U)^\ast,
\]
as desired.
\end{proof}

\section{Symplectic $\mathcal A$-modules. Witt's theorem}

In this section, for the sake of self-containedness of the paper, we
first recall the notion of symplectic $\mathcal A$-modules, and then
describe how to characterize symplectomorphic $\mathcal A$-modules.
We refer the reader to \cite{darboux} and \cite{orthogonally} for
useful details regarding symplectic $\mathcal A$-modules and
symplectic bases (of sections). Sheaves of symplectic groups arise
in a natural way when one considers $\mathcal A$-isomorphisms
between symplectic $\mathcal A$-modules which respect the symplectic
structures involved. Finally, the section ends with a version of
Witt's theorem for symplectic $\mathcal A$-modules. For some other
versions of the Witt's theorem, see \cite{malliosntumba2} and
\cite{ntumbawitt}.

\begin{definition}
\emph{Let $\mathcal E$ be a \textit{free $\mathcal A$-module of
finite rank,} endowed with a \textit{skew-symmetric non-degenerate
$\mathcal A$-bilinear morphism} $\omega: \mathcal{E}\oplus
\mathcal{E}\longrightarrow \mathcal A$. Then, the pair
$(\mathcal{E}, \omega)$ is called a \textbf{symplectic $\mathcal
A$-module}.}
\end{definition}

\begin{definition}
\emph{Let $(\mathcal{E}, \omega)$ and $(\mathcal{E}', \omega')$ be
symplectic $\mathcal A$-modules. An $\mathcal A$-morphism
$\varphi\in \mbox{H}om_\mathcal{A}(\mathcal{E}, \mathcal{E}')$ is
called \textbf{symplectic} if
\[
\varphi^\ast\omega':= \omega'\circ (\varphi\times \varphi)= \omega,
\]
that is, for any sections $s, t\in \mathcal{E}(U),$
\[
(\varphi^\ast_U\omega')(s, t):= \omega'_U\circ (\varphi_U(s),
\varphi_U(t))= \omega_U(s, t).
\]
A symplectic $\mathcal A$-isomorphism is called an \textbf{$\mathcal
A$-symplectomorphism}. Symplectic $\mathcal A$-modules $\mathcal{E},
\omega)$ and $(\mathcal{E}', \omega')$ are called symplectomorphic
if there is an $\mathcal A$-symplectomorphism $\varphi$ between
them. }
\end{definition}

The following result is not hard to prove (see \cite[pp. 187-189,
Lemma 4]{darboux} for a proof thereof), and introduces a particular
case of the notion of symplectic group sheaf (or sheaf of symplectic
groups).

\begin{lemma}\label{lemma2}
Let $\mathcal E\equiv (\mathcal{E}, \omega)$ be a symplectic
$\mathcal A$-module and let
\[
(Sp\ \mathcal{E})(U)\subseteq Aut_{\mathcal{A}|_U}(\mathcal{E}|_U),
\]
where $U$ varies over the topology of $X$, be the group $($under
composition$)$ of all $\mathcal{A}|_U$-symplectomorphisms of
$\mathcal{E}|_U$ into $\mathcal{E}|_U$. Then, mappings
\[
U\longmapsto (Sp\ \mathcal{E})(U)
\]
together with the obvious restriction maps yield a complete presheaf
of groups on $X$. If $\mathcal{S}p\ \mathcal E$ the sheaf on $X$,
generated by the aforesaid presheaf, one has
\[
(\mathcal{S}p\ \mathcal{E})(U)= (Sp\ \mathcal{E})(U)
\]
up to a group isomorphism, for every open $U\subseteq X$. The sheaf
$\mathcal{S}p\ \mathcal E$ is called the \textsf{symplectic group
sheaf} of $\mathcal E$ $($in fact, of $(\mathcal{E}, \omega)$)
\end{lemma}

For an example of an $\mathcal A$-symplectic form, consider the
$\mathcal A$-bilinear form, denoted $(\ |\ ),$ on the standard free
$\mathcal A$-module $\mathcal{A}^{2n}$ such that, given any open set
$U$ in $X$ and sections $a\equiv (a_1, \ldots, a_{2n}),$ $b\equiv
(b_1, \ldots, b_{2n})\in \mathcal{A}^{2n}(U)= \mathcal{A}(U)^{2n},$
one has

\[
(a\ |\ b):= \sum_{i=1}^na_i b_{i+n}- a_{i+n}b_i.
\]
The $\mathcal A$-bilinear form $(\ |\ )$ is called the
\textbf{standard $\mathcal A$-symplectic scalar product} or
\textbf{standard $\mathcal A$-symplectic form} on
$\mathcal{A}^{2n}$, and the pair $(\mathcal{A}^{2n}, (\ |\ ))$ the
\textbf{standard symplectic $\mathcal A$-module of rank $2n$.} When
there is no confusion about the symplectic $\mathcal A$-structure,
the standard symplectic $\mathcal A$-module $(\mathcal{A}^{2n}, (\
|\ ))$ will simply be denoted by $\mathcal{A}^{2n}$. The symplectic
group sheaf of the standard symplectic $\mathcal A$-module
$\mathcal{A}^{2n}$ is denoted by $\mathcal{S}p (2n; \mathcal{A})$
(or just $\mathcal{S}p (n; \mathcal{A})$.

An element of $\mathcal{S}p(2n; \mathcal{A})(U)$, where $U$ is an
open subset of $X$, is called a \textit{symplectic section-matrix.}

Recall (see \cite{darboux}, \cite{orthogonally}) that if a pair
$(\mathcal{E}, \omega)$ is a symplectic $\mathcal A$-module, then
given any open subset $U$ of $X$, $\mathcal{E}(U)$ may be equipped
with a basis $s_1, \ldots, s_n, t_1, \ldots, t_n$ with respect to
which the symplectic form $\omega_U$ is represented by the matrix
\[
J= \left(\begin{array}{cc} 0 & \mbox{I}_n\\ -\mbox{I}_n & 0
\end{array}\right);
\]
in other words
\[
\begin{array}{llll} \omega_U(s_i, s_j)= \omega_U(t_i, t_j)=0, &
\omega_U(s_i, t_j)= \delta_{ij} & & \mbox{for $1\leq i, j\leq n$}
\end{array}
\]
A basis $(s_1, \ldots, s_n, t_1, \ldots, t_n)$ is called a
\textit{symplectic basis} of $\mathcal{E}(U).$

If $(\mathcal{E}, \omega)$ is a symplectic $\mathcal A$-module, and
$A$ and $B$ two column (coordinate) matrices representing sections
$s$ and $t$, respectively, with respect to a symplectic basis of
$\mathcal{E}(U),$ then
\[
\omega_U(s, t)= {}^tAJB,
\]
where ${}^tA$ denotes the transpose of $A$.

Applying the classical symplectic algebra machinery (see for
instance \cite[pp. 407, 408]{chambadal} and \cite[pp.
410-413]{deheuvels}) and in view of Lemma \ref{lemma2}, a $2n\times
2n$ \textit{section-matrix} $M\equiv \left(\begin{array}{ll} M_{11}
& M_{12}\\ M_{21} & M_{22}\end{array}\right)\in
\mathcal{G}\mathcal{L}(2n, \mathcal{A})(U)= GL_{2n}(\mathcal{A})(U)=
GL_{2n}(\mathcal{A}(U))$ (where $\mathcal{G}\mathcal{L}(2n,
\mathcal{A})$ is the general linear group sheaf generated by the
(complete) sub-presheaf
\[
U\longmapsto GL(2n, \mathcal{A})(U)= GL_{2n}(\mathcal{A})(U)=
GL_{2n}(\mathcal{A}(U))
\]
of the (full) matrix algebra sheaf $M_n(\mathcal{A}),$ cf. \cite[pp.
280-285]{mallios}) \textit{is symplectic} if and only if for any
elements (sections) $A$ and $B$ of $(\mathcal{A}^{2n}(U), (\ |\ ))=
(\mathcal{A}(U)^{2n}, (\ |\ ))$
\[
(MA|\ MB)= {}^t(MA)J MB= {}^tA{}^tMJMB= {}^tAJB,
\]
that is
\begin{equation}\label{1.10}
{}^tMJM= J.
\end{equation}
Equation (\ref{1.10}) splits into three:
\begin{enumerate}
\item ${}^tM_{11}M_{21}= {}^tM_{21}M_{11}$ \hspace{5mm}(i.e.
${}^tM_{11}M_{21}$ is symmetric), \item ${}^tM_{12}M_{22}=
{}^tM_{22}M_{12}$ \hspace{5mm}(i.e. ${}^tM_{12}M_{22}$ is
symmetric),
\item $-{}^tM_{21}M_{12}+ {}^tM_{11}M_{22}= I_n.$
\end{enumerate}

We will now prove an analog of the Witt's theorem within the context
of Abstract Differential Geometry. For the classical Witt's theorem,
see \cite[pp. 368-387]{adkins}, \cite[pp. 121, 122]{artin}, \cite[p.
21]{berndt}, \cite{dasilva}, \cite[pp. 11, 12]{crumeyrolle},
\cite[pp. 148- 152]{deheuvels}, \cite[pp. 591, 592]{lang}, \cite[p.
9]{omeara}. But, first we need the following definition (cf.
\cite{cartandieudonne}).

\begin{definition}
\emph{A sheaf of algebras $\mathcal A$ is called a
\textbf{PID-algebra sheaf} if for every open $U\subseteq X$, the
algebra $\mathcal{A}(U)$ is a PID-algebra. In other words, given a
free $\mathcal A$-module $\mathcal E$ and a sub-$\mathcal A$-module
$\mathcal{F}\subseteq \mathcal E$, one has that $\mathcal F$ is
``\textit{section-wise free.}" That is, $\mathcal{F}(U)$ is a
\textit{free $\mathcal{A}(U)$-module,} for any open $U\subseteq X$.}
\end{definition}

\begin{proposition}\label{proposition1}
Let $\mathcal A$ be a \textsf{PID algebra sheaf}, $(\mathcal{E},
\omega)$ a symplectic free $\mathcal A$-module of rank $2n$,
$\mathcal F$ a \textsf{Lagrangian $($free$)$ sub-$\mathcal
A$-module} of $\mathcal E$ and $\mathcal G$ any sub-$\mathcal
A$-module of $\mathcal E$ such that $\mathcal F$ and $\mathcal G$
are supplementary. Then, using $\mathcal G$ we can construct a
\textsf{Lagrangian sub-$\mathcal A$-module} $\mathcal H$ of
$\mathcal E$ such that $\mathcal{E}\cong \mathcal{F}\oplus
\mathcal{H}.$
\end{proposition}

\begin{proof}
The restriction $\omega'$ of $\omega$ to $\mathcal{F}\oplus
\mathcal{G}\subseteq \mathcal{E}\oplus \mathcal{E}$ is also
non-degenerate. In fact, let $\mathcal{F}^\perp_{\omega'}$ and
$\mathcal{G}^\perp_{\omega'}$ denote the kernels of $\mathcal F$ and
$\mathcal G$ respectively. More precisely, for every open
$U\subseteq X$,
\[\mathcal{F}^\perp_{\omega'}(U)= \{r\in \mathcal{G}(U)|\
\omega'(\mathcal{F}(V), r|_V)=0 \ \mbox{for any open $V\subseteq
U$}\}\] and similarly
\[\mathcal{G}^\perp_{\omega'}(U)= \{r\in \mathcal{F}(U)|\
\omega'(r|_V, \mathcal{G}(V))=0 \ \mbox{for any open $V\subseteq
U$}\}.\] Analogously we denote by $\mathcal{F}^\perp_\omega$ and
$\mathcal{G}^\perp_\omega$ the kernels of $\mathcal F$ and $\mathcal
G$ respectively with respect to the $\mathcal A$-bilinear morphism
$\omega: \mathcal{E}\oplus \mathcal{E}\longrightarrow \mathcal{A},$
i.e. for every open $U\subseteq X$, \[\mathcal{F}_\omega^\perp(U)=
\{r\in \mathcal{E}(U)|\ \omega(\mathcal{F}(V), r|_V)=0\ \mbox{for
any open $V\subseteq U$}\}\] and  \[\mathcal{G}_\omega^\perp(U)=
\{r\in \mathcal{E}(U)|\ \omega(\mathcal{G}(V), r|_V)=0\ \mbox{for
any open $V\subseteq U$}\}.\] It is obvious that
$\mathcal{F}_\omega^\perp= \mathcal{F}_\omega^\top$ and
$\mathcal{G}_\omega^\perp= \mathcal{G}_\omega^\top.$ By hypothesis,
we are given that $\mathcal{F}= \mathcal{F}^\perp_\omega.$ Clearly,
for every open $U\subseteq X$,
$\mathcal{F}^\perp_{\omega'}(U)\subseteq
\mathcal{F}^\perp_\omega(U)$ and
$\mathcal{G}^\perp_{\omega'}(U)\subseteq
\mathcal{G}^\perp_\omega(U).$ But since
$\mathcal{F}^\perp_\omega(U)= \mathcal{F}(U)$ and
$\mathcal{F}(U)\cap \mathcal{G}(U)=0,$
$\mathcal{F}^\perp_{\omega'}(U)=0.$ Thus,
$\mathcal{F}^\perp_{\omega'}=0.$ On the other hand, let $r\in
\mathcal{G}^\perp_{\omega'}(U)\subseteq \mathcal{F}(U)\cap
\mathcal{G}^\perp_\omega(U).$ As $\mathcal{E}(U)=
\mathcal{F}(U)\oplus \mathcal{G}(U)$, we deduce that $r\in \mbox{rad
$\mathcal{E}(U)$}= 0$, therefore $r=0.$ Hence,
$\mathcal{G}^\perp_{\omega'}=0.$ Since $\omega': \mathcal{F}\oplus
\mathcal{G}\longrightarrow \mathcal A$ is non-degenerate, the
$\mathcal A$-morphism $\widetilde{\omega'}:
\mathcal{F}\longrightarrow \mathcal{G}^\ast$ such that for every
open $U\subseteq X$, and sections $r\in \mathcal{F}(U)$ and $s\in
\mathcal{G}(U),$ $\widetilde{\omega'}(r)(s):= \omega'(r, s)$ is
bijective.

Let us construct the sought Lagrangian complement $\mathcal H$ of
$\mathcal F$ in $\mathcal E$. For every open $U\subseteq X$, we let
\[\mathcal{H}(U):= \{r+ \phi(r)|\ r\in \mathcal{G}(U)\},\]where
$\phi: \mathcal{G}\longrightarrow \mathcal F$ is some $\mathcal
A$-morphism. It is clear that $\mathcal H$ is a sub-$\mathcal
A$-module of $\mathcal E$. For $\mathcal H$ to be Lagrangian, it
takes the following: For every open $U\subseteq X$ and sections $r,
s\in \mathcal{G}(U)$
\[\omega(r+ \phi(r), s+ \phi(s))=0\]i.e.
\begin{equation}\label{eq7}
\omega(r, s)= \widetilde{\omega'}(\phi(s))(r)-
\widetilde{\omega'}(\phi(r))(s).
\end{equation}
Let $\phi':= \widetilde{\omega'}\circ \phi:
\mathcal{G}\longrightarrow \mathcal{G}^\ast,$ so that (\ref{eq7})
becomes
\begin{equation}\label{eq8}
\omega(r, s)= \phi'(s)(r)- \phi'(r)(s).
\end{equation}
Clearly, by taking $\phi'(r)= -\frac{1}{2}\omega(r, -)$ for every
$r\in \mathcal{G}(U),$ (\ref{eq8}) is satisfied. By setting $\phi:=
(\widetilde{\omega'})^{-1}\circ \phi',$ we contend that the claim
holds. In fact, fix an open subset $U$ of $X$, and suppose that
$(r_1, \ldots, r_n)$ is a basis of $\mathcal{G}(U).$ If $a_1,
\ldots, a_n\in \mathcal{A}(U)$ such that \[a_1(r_1+ \phi(r_1))+
\ldots + a_n(r_n+ \phi(r_n))=0,\] one has that
\[\underbrace{a_1r_1+\ldots + a_nr_n}_{\in \mathcal{G}(U)}= \underbrace{-\phi(a_1r_1+ \ldots +
a_nr_n)}_{\in \mathcal{F}(U)}.\] Since $\mathcal{F}(U)\cap
\mathcal{G}(U)=0,$ it follows that \[\phi(a_1r_1+ \ldots +
a_nr_n)=0.\]As the chosen $\phi'$ is injective and
$\widetilde{\omega'}$ is an $\mathcal A$-isomorphism, $\phi$ is
injective; thence
\[a_1r_1+ \ldots + a_nr_n=0;\]so that $a_1=\cdots = a_n=0.$ Now, let
us show that $\mathcal{F}(U)\cap \mathcal{H}(U)=0.$ For this
purpose, suppose that $r\in \mathcal{F}(U)\cap \mathcal{H}(U).$ Then
for some $s\in \mathcal{G}(U)$ \[r= s+ \phi(s).\] It follows that
\[\underbrace{r-\phi(s)}_{\in \mathcal{F}(U)}= \underbrace{s}_{\in
\mathcal{G}(U)}\] from which we deduce that $s=0,$ and hence $r=0$.
That $\mathcal{E}(U)\cong \mathcal{F}(U)\oplus \mathcal{H}(U)$ is
now clear. Since $U$ is arbitrary, $\mathcal{E}\cong
\mathcal{F}\oplus \mathcal{H}$ as desired.
\end{proof}

\begin{theorem}\ $($\textbf{Witt's Theorem}$)$
Let $\mathcal A$ be a PID algebra sheaf, let $\mathcal E$ be a free
$\mathcal A$-module of rank $2n$, equipped with two symplectic
$\mathcal A$-morphisms $\omega_0$ and $\omega_1$, and finally let
$\mathcal F$ be a sub-$\mathcal A$-module of $\mathcal E$,
\textsf{Lagrangian with respect to both $\omega_0$ and $\omega_1$.}
Then, there exists an \textsf{$\mathcal A$-symplectomorphism} $\phi:
(\mathcal{E}, \omega_0)\longrightarrow (\mathcal{E}, \omega_1)$ such
that $\phi|_\mathcal{F}= \mbox{Id}_\mathcal{F}.$
\end{theorem}

\begin{proof} Let $\mathcal G$ be any complement of $\mathcal F$ in
$\mathcal E$. By Proposition \ref{proposition1}, given symplectic
$\mathcal A$-morphisms $\omega_0$ and $\omega_1$, there exist
Lagrangian complements $\mathcal{G}_0$ and $\mathcal{G}_1$ of
$\mathcal F$ respectively. Again by the proof of Proposition
\ref{proposition1}, the restrictions $\omega_0'$, $\omega_1'$ of
$\omega_0$, $\omega_1$ to $\mathcal{G}_0\oplus \mathcal F$ and
$\mathcal{G}_1\oplus \mathcal F$ respectively are nondegenerate and
yield $\mathcal A$-isomorphisms $\widetilde{\omega_0'}:
\mathcal{G}_0\longrightarrow \mathcal{F}^\ast$ and
$\widetilde{\omega_1'}: \mathcal{G}_1\longrightarrow
\mathcal{F}^\ast$ respectively. Since $\mathcal{G}_0$ and
$\mathcal{G}_1$ are free and of the same finite rank, there exists
an $\mathcal A$-isomorphism $\psi: \mathcal{G}_0\longrightarrow
\mathcal{G}_1$ such that $\widetilde{\omega_1'}\circ \psi=
\widetilde{\omega_0'},$ i.e. for any sections $r\in
\mathcal{G}_0(U)$ and $s\in \mathcal{F}(U)$
\[\omega_0(r, s)= \omega_1(\psi(r), s).\] Let us extend $\psi$ to
the rest of $\mathcal E$ by setting it to be the identity on
$\mathcal F$: \[\phi:= \mbox{Id}_\mathcal{F}\oplus \psi:
\mathcal{F}\oplus \mathcal{G}_0\longrightarrow \mathcal{F}\oplus
\mathcal{G}_1\] and we have for any sections $r, r'\in
\mathcal{G}_0(U)$ and $s, s'\in \mathcal{F}(U)$ \[\begin{array}{lll}
\omega_1(\phi(s+r), \phi(s'+ r')) & = & \omega_1(s+ \psi(r), s'+
\psi(r')) \\ & = & \omega_1(s, \psi(r'))+ \omega_1(\psi(r), s')\\ &
= & \omega_0(s, r') + \omega_0(r, s')\\ & = & \omega_0(s+r, s'+
r').\end{array}\]
\end{proof}

\section{Orthogonally convenient pairings}

We introduced here a new subclass of $\mathcal A$-modules: the
\textit{orthogonally convenient pairings} of \textit{$\mathcal
A$-modules,} with the aim of achieving the characterization of
singular symplectic $\mathcal A$-automorphisms of symplectic
orthogonally convenient $\mathcal A$-modules of finite rank.

We now make the following two definitions.

\begin{definition}
\emph{A pairing $[\mathcal{F}, \mathcal{E}; \mathcal{A}]$ of free
$\mathcal A$-modules $\mathcal F$ and $\mathcal E$ into the
$\mathbb{C}$-algebra sheaf $\mathcal A$ is called an
\textbf{orthogonally convenient pairing} if given free sub-$\mathcal
A$-modules $\mathcal{F}_0$ and $\mathcal{E}_0$ of $\mathcal F$ and
$\mathcal E$, respectively, their orthogonal $\mathcal{F}_0^\perp$
and $\mathcal{E}_0^\perp$ are free sub-$\mathcal A$-modules of
$\mathcal E$ and $\mathcal F$, respectively. }\end{definition}

\begin{definition}
\emph{The pairing $[\mathcal{E}^\ast, \mathcal{E};
\mathcal{A}]\equiv [(\mathcal{E}^\ast, \mathcal{E}; \phi);
\mathcal{A}]$, where $\mathcal E$ is a free $\mathcal A$-module and
such that for every open $U\subseteq X$, \[\phi_U(\psi, r):=
\psi_U(r)\]with $\psi\in
\mathcal{E}^\ast(U):=\mbox{H}om_{\mathcal{A}|_U}(\mathcal{E}|_U,
\mathcal{A}|_U)$ and $r\in \mathcal{E}(U)$, is called the
\textbf{canonical pairing} of $\mathcal{E}^\ast$ and $\mathcal E$.
}\end{definition}

\begin{theorem}\label{theo4}
Let $\mathcal E$ be a free $\mathcal A$-module of finite rank. The
canonical pairing $[(\mathcal{E}^\ast, \mathcal{E}; \phi);
\mathcal{A}]$ is orthogonally convenient.
\end{theorem}

\begin{proof}
First, we notice by Theorem \ref{theo2} that both kernels, i.e.
$(\mathcal{E}^\ast)^\perp$ and $\mathcal{E}^\perp$, are $0$. Let
$\mathcal{E}_0$ be a free sub-$\mathcal A$-module of $\mathcal E$,
and consider the map (\ref{eq10}) of Theorem \ref{theo3}$:
\mathcal{E}_0^\perp\longrightarrow
(\mathcal{E}/\mathcal{E}_0)^\ast.$ It is an $\mathcal A$-isomorphism
into, and we shall show that it is onto. Fix an open set $U$ in $X$,
and let $\psi\in (\mathcal{E}/\mathcal{E}_0)^\ast(U):=
\mbox{H}om_{\mathcal{A}|_U}\left((\mathcal{E}/\mathcal{E}_0)|_U,
\mathcal{A}|_U\right).$ Let us consider a family
$\overline{\psi}\equiv (\overline{\psi}_V)_{U\supseteq V,\ open}$
such that
\begin{equation}\begin{array}{ll}\overline{\psi}_V(r):= \psi_V(r+
\mathcal{E}_0(V)), & r\in \mathcal{E}(V).
\label{eq18}\end{array}\end{equation} It is easy to see that
$\overline{\psi}_V$ is $\mathcal{A}(V)$-linear for any open
$V\subseteq U$. Now, let $\{\rho^V_W\}$, $\overline{\rho}^V_W$ and
$\{\tau^V_W\}$ be the restriction maps for the (complete) presheaves
of sections of $\mathcal E$, $\mathcal{E}/\mathcal{E}_0$ and
$\mathcal A$, respectively. The restriction maps
$\overline{\rho}^V_W$ are defined by setting
\[\begin{array}{ll}
\overline{\rho}^V_W(r+ \mathcal{E}_0(V)):= \rho^V_W(r)+
\mathcal{E}_0(W), & r\in \mathcal{E}(V)\end{array}.
\]
It clearly follows that
\begin{eqnarray*}
(\tau^V_W\circ \overline{\psi}_V)(r) & = & \tau^V_W(\psi_V(r+
\mathcal{E}_0(V)))\\ & = & \psi_W(\rho^V_W(r)+ \mathcal{E}_0(W))\\ &
= & \overline{\psi}_W(\rho^V_W(r))\\ & = & (\overline{\psi}_W\circ
\rho^V_W)(r),
\end{eqnarray*}
from which we deduce that
\[
\tau^V_W\circ \overline{\psi}_V= \overline{\psi}_W\circ \rho^V_W,
\]
which implies that
\[
\overline{\psi}\in \mbox{H}om_{\mathcal{A}|_U}(\mathcal{E}|_U,
\mathcal{A}|_U)=: \mathcal{E}^\ast(U).
\] Suppose $r\in
\mathcal{E}_0(V)$, where $V$ is open in $U$. Then
\[\overline{\psi}_V(r)= \psi_V(r+ \mathcal{E}_0(V))=
\psi_V(\mathcal{E}_0(V))= 0,\]therefore
\[\phi_V(\overline{\psi}|_V, \mathcal{E}_0(V))=
\overline{\psi}_V(\mathcal{E}_0(V))=0,\]i.e. $\overline{\psi}\in
\mathcal{E}_0^\perp(U).$ We contend that $\overline{\psi}$ has the
given $\psi$ as image under the map (\ref{eq10}), and this will show
the ontoness of (\ref{eq10}) and that $\mathcal{E}_0^\perp$ is a
free sub-$\mathcal A$-module of $\mathcal{E}^\ast$.

Let us find the image of $\overline{\psi}$. Consider the pairing
$[(\mathcal{E}_0^\perp, \mathcal{E}/\mathcal{E}_0; \Theta);
\mathcal{A}]$ such that for any open $V\subseteq X$, we have
\[\Theta_V(\alpha, r+ \mathcal{E}_0(V)):= \phi_V(\alpha, r)=
\alpha_V(r),\]where $\alpha\in \mathcal{E}_0^\perp(V)\subseteq
\mathcal{E}^\ast(V),$ $r\in \mathcal{E}(V)$. Clearly, the left
kernel of this new pairing is $0$. For $\alpha= \overline{\psi}\in
\mathcal{E}_0^\perp(U)\subseteq \mathcal{E}^\ast(U),$ we have
\[\Theta_U(\overline{\psi}, r+ \mathcal{E}_0(U))=
\overline{\psi}_U(r)\]where $r\in \mathcal{E}(U)$, and the map
\[\overline{\Theta}_U: \mathcal{E}_0^\perp(U)\longrightarrow
(\mathcal{E}/\mathcal{E}_0)^\ast(U)\]given by
\[\overline{\psi}\longmapsto \overline{\Theta}_{U,
\overline{\psi}}\equiv \left((\overline{\Theta}_{U,
\overline{\psi}})_V\right)_{U\supseteq V,\ open}\] and such that for
any $r\in \mathcal{E}(V)$ \[(\overline{\Theta}_{U,
\overline{\psi}})_V(r+ \mathcal{E}_0(V)):=
\overline{\Theta}_V(\overline{\psi}|_V, r+\mathcal{E}_0(V))=
\overline{\psi}_V(r)= \psi_V(r+ \mathcal{E}_0(V))\]is the image.
Thus the image of $\overline{\psi}$ is $\psi$, hence the map
$\mathcal{E}_0^\perp(U)\longrightarrow
(\mathcal{E}/\mathcal{E}_0)^\ast(U),$ derived from (\ref{eq10}), is
onto, and therefore an $\mathcal{A}(U)$-isomorphism. Since
$\mathcal{E}/\mathcal{E}_0$ is free by Corollary \ref{corollary1},
so are $(\mathcal{E}/\mathcal{E}_0)^\ast$ and $\mathcal{E}_0^\perp$
free.

Now, let $\mathcal{F}_0$ be a free sub-$\mathcal A$-module of
$\mathcal{E}^\ast\cong \mathcal E$ (cf. Mallios~\cite[p.298,
(5.2)]{mallios}); on considering $\mathcal{F}_0$ as a free
sub-$\mathcal A$-module of $\mathcal E$, according to all that
precedes above $\mathcal{F}_0^\perp$ is free in
$\mathcal{E}^\ast\cong \mathcal E$, and so the proof is finished.
\end{proof}

We now make the following important observation concerning
symplectic $\mathcal A$-modules of finite rank.

\begin{lemma}\label{lemma2}
Let $(\mathcal{E}, \omega)$ be a symplectic $\mathcal A$-module of
finite rank, and $f\equiv (f_U)_{X\supseteq U,\ open}$ an $\mathcal
A$-endomorphism of $\mathcal E$. Then, if $f$ satisfies two of the
three following conditions, it satisfies all of them three:
\begin{enumerate}
\item [{$(1)$}] $I+f$ is an $\mathcal A$-automorphism of $\mathcal
E$;
\item [{$(2)$}] $f$ is $\omega$-skewsymmetric, i.e., for any open
$U\subseteq X$ and sections $s, t\in \mathcal{E}(U),$
\[
\omega_U(f_U(s), t)+ \omega_U(s, f_U(t))= 0;
\]
\item [{$(3)$}] $\mbox{Im}\ f\equiv f(\mathcal{E})$ is totally
isotropic, i.e., for any open $U\subseteq X$ and sections $s, t$ of
$f(\mathcal{E})$ on $U$,
\[
\omega_U(s, t)=0.
\]
\end{enumerate}
\end{lemma}

\begin{proof}
Using the equality
\[
\omega_U(s+ f_U(s), t+ f_U(t))- \omega_U(s, t)= \omega_U((f_U+
f^\ast_U)(s), t)+ \omega_U(f_U(s), f_U(t)),
\]
where $U$ is any open subset of $X$, $s$ and $t$ sections of
$\mathcal E$ over $U$, one easily checks the implications: $(1),\
(2)\Rightarrow (3)$; $(1),\ (3)\Rightarrow (2)$; and $(2),\
(3)\Rightarrow (1).$
\end{proof}

Now, we are going to introduce a new class of $\mathcal A$-modules
we will be concerned with in the sequel.

\begin{definition}
\emph{An $\mathcal A$-module $\mathcal E$ is called a \textbf{locally free
$\mathcal A$-module of varying finite rank} if there exist an open
covering $\mathcal{U}\equiv (U_\alpha)_{\alpha\in I}$ of $X$ and
numbers $n(\alpha)\in \mathbb{N}$ for every open set $U_\alpha$ such
that
\[
\mathcal{E}|_{U_\alpha}= \mathcal{A}^{n(\alpha)}|_{U_\alpha}.
\]
The open covering $\mathcal{U}$ is called a \textbf{local frame.}}
\end{definition}

\begin{example}
\emph{Consider a free $\mathcal A$-module $\mathcal E$, where
$\mathcal A$ is a PID-algebra sheaf. Then, every sub-$\mathcal
A$-module of $\mathcal E$ is a locally free $\mathcal A$-module of
varying finite rank.}
\end{example}

We come now to the following useful result, satisfied by free
$\mathcal A$-modules of finite rank (cf. \cite[p. 298,
(5.2)]{mallios}) and vector sheaves of finite rank (cf. \cite[p.
138, (6.26)]{mallios}). That is, one has:

\begin{lemma}\label{lemma3}
Let $\mathcal E$ be a locally free $\mathcal A$-module of varying
finite rank. Then, the dual $\mathcal A$-module $\mathcal{E}^\ast:=
\mathcal{H}om_\mathcal{A}(\mathcal{E}, \mathcal{A})$ is a locally
free $\mathcal A$-module of varying finite rank, and one has
\[
\mathcal{E}^\ast= \mathcal{E}
\]
within an $\mathcal A$-isomorphism.
\end{lemma}

\begin{proof}
Let $\mathcal{U}\equiv (U_\alpha)_{\alpha\in I}$ be a local frame of
$\mathcal E$. Applying \cite[p. 137, (6.22) and (6.23)]{mallios},
one has, for every $\alpha\in I,$
\begin{eqnarray*}
\mathcal{H}om_\mathcal{A}(\mathcal{E}, \mathcal{A})|_{U_\alpha} & =
& \mathcal{H}om_{\mathcal|_{U_\alpha}}(\mathcal{E}|_{U_\alpha},
\mathcal{A}|_{U_\alpha})\\ & = &
\mathcal{H}om_{\mathcal{A}|_{U_\alpha}}(\mathcal{A}^{n(\alpha)}|_{U_\alpha},
\mathcal{A}|_{U_\alpha})\\ & = &
\mathcal{H}om_\mathcal{A}(\mathcal{A}^{n(\alpha)},
\mathcal{A})|_{U_\alpha}\\ & = & \mathcal{A}^{n(\alpha)},
\end{eqnarray*}
within $\mathcal{A}|_{U_\alpha}$-isomorphisms.
\end{proof}

Our objective now is to obtain a useful characterization of a
symplectic $\mathcal A$-automorphism of the form $I+f$ of a
symplectic orthogonally convenient $\mathcal A$-module $\mathcal E$,
where $f$ is a skewsymmetric $\mathcal A$-endomorphism of $\mathcal
E$. For this purpose, we require a generalization of the following
result, see \cite{malliosntumba2}.

\begin{theorem}\label{theorem3}
Let $(\mathcal{E}, \phi)$ be a free $\mathcal{A}$-module of finite
rank. Then, every non-isotropic free sub-$\mathcal{A}$-module
$\mathcal{F}$ of $\mathcal{E}$ is a direct summand of $\mathcal{E}$;
viz. \[\mathcal{E}= \mathcal{F}\bot\ \mathcal{F}^\perp.\]
\end{theorem}

We generalize Theorem \ref{theorem3} as follows:

\begin{theorem}\label{theorem5}
Let $(\mathcal{E}, \phi)$ be a free $\mathcal A$-module of finite
rank, with $\mathcal A$ a PID-algebra sheaf. Then, every
non-isotropic sub-$\mathcal A$-module $\mathcal F$ of $\mathcal E$
is a direct summand of $\mathcal E$; viz.
\[
\mathcal{E}= \mathcal{F}\bot \mathcal{F}^\perp
\]
within an $\mathcal A$-isomorphism.
\end{theorem}

\begin{proof}
First, we notice that $\mathcal F$ is a locally free sub-$\mathcal
A$-module of varying finite rank. Then, let us consider for any open
subset $U\subseteq X$ a section $t\in \mathcal{E}(U)$ and a section
$\psi\in \mathcal{F}^\ast(U):=
\mathcal{H}om_\mathcal{A}(\mathcal{F}, \mathcal{A})(U)\equiv
Hom_{\mathcal{A}|_U}(\mathcal{F}|_U, \mathcal{A}|_U),$ defined as
follows: given $s\in \mathcal{F}(V),$ where $V$ is open in $U$,
\[
\psi_V(s):= \phi_V(t|_V, s).
\]
$\mathcal F$ being non-isotropic, we have that the restriction
$\phi|_{\mathcal F}$ of $\phi$ on $\mathcal F$ is non-degenerate,
consequently, since $\mathcal{F}^\ast\cong \mathcal F$ (cf. Lemma
\ref{lemma3}), $\psi$ may be identified with a unique element
(section) $p_U(t)\equiv p(t)\in \mathcal{F}(U)\cong
\mathcal{F}^\ast(U)$ in such a way that
\[
\phi_V(t|_V, s)=
(\phi|_\mathcal{F})_V(p(t)|_V, s)= \phi_V(p(t)|_V, s)
\]
for all $s\in \mathcal{F}(V)$. it is easily seen that $p(\alpha r+
t)= \alpha p(r)+ p(t),$ for $r, t\in \mathcal{E}(U)$ and $\alpha\in
\mathcal{A}(U)$, that is $p: \mathcal{E}(U)\longrightarrow
\mathcal{E}(U)$ is $\mathcal{A}(U)$-linear. Next, since $p^2= p$,
then $p: \mathcal{E}(U)\longrightarrow \mathcal{F}(U)$ is an
$\mathcal{A}(U)$-projection. Furthermore, since
\[\phi_V((t-p(t))|_V, s)= \phi_V(t|_V- p(t)|_V, s)=0\] for all
$t\in \mathcal{E}(U)$ and $s\in \mathcal{F}(V)$, with $V$ open in
$U$, the supplementary $\mathcal{A}(U)$-projection $q:= I-p$ is such
that for all $t\in \mathcal{E}(U)$, $q(t)\equiv (I-p)(t)\in
\mathcal{F}^\perp(U)$, i.e. $q$ maps $\mathcal{E}(U)$ on
$\mathcal{F}^\perp(U)$. Hence, every element $t\in \mathcal{E}(U)$,
where $U$ runs over the open subsets of $X$, may be written as \[t=
p(t)+ (t- p(t))\] with $p(t)\in \mathcal{F}(U)$ and $t-p(t)\in
\mathcal{F}^\perp(U).$ Thus \[\mathcal{E}(U)= \mathcal{F}(U)\oplus
\mathcal{F}^\perp(U)= (\mathcal{F}\oplus \mathcal{F}^\perp)(U)\]
within $\mathcal{A}(U)$-isomorphisms (for the
$\mathcal{A}(U)$-isomorphism $\mathcal{F}(U)\oplus
\mathcal{F}^\perp(U)= (\mathcal{F}\oplus \mathcal{F}^\perp)(U),$ cf.
Mallios[\cite{mallios}, relation (3.14), p. 122]). Finally, since
$\mathcal{F}$ is non-isotropic, it follows that
\[\mathcal{E}(U)= (\mathcal{F}\bot \mathcal{F}^\perp)(U)\]for
every open $U\subseteq X$. Thus, we reach the sought
$\mathcal{A}$-isomorphism $\mathcal{E}= \mathcal{F}\bot
\mathcal{F}^\perp.$

\end{proof}

In keeping with the notations of Theorem \ref{theorem5}, we clearly
have that: $ \mathcal{F}^{\perp\perp}\cong \mathcal{F}.$ Moreover,
if $\phi$ is \textit{non-degenerate,} then
$\mathcal{F}^\perp(U)\cong \mathcal{F}(U)^\perp,$ for every open
$U\subseteq X.$ Indeed, since $\mathcal{F}^\perp(U)\subseteq
\mathcal{F}(U)^\perp$, then if $\mathcal{F}(U)\cap
\mathcal{F}(U)^\perp\neq 0$, rad $\mathcal{E}(U)\neq 0$, which
\textit{contradicts} the hypothesis that $\mathcal E$ is
non-isotropic.

\begin{theorem}\label{theorem4}
Let $(\mathcal{E}, \omega)$ be a symplectic orthogonally convenient
$\mathcal A$-module of rank $2n$, and $f$ a $\mathcal
A$-endomorphism of $\mathcal E$. If $f$ is skewsymmetric and $I+ f$
an $ \mathcal{A}$-automorphism of $\mathcal E$, then
\begin{enumerate}
\item [{$(\mathbf{1})$}] $f^2=0;$
\item [{$(\mathbf{2})$}] $\ker f\simeq (\mbox{Im}\ f)^\perp;$
\item [{$(\mathbf{3})$}] For every open subset $U\subseteq X$, there exists a
symplectic basis of $\mathcal{E}(U),$ whose first $k$ elements
$($sections$)$, $k\leq n,$ form a basis of $(\mbox{Im}\ f)(U):=
\mbox{Im}\ f_U\equiv f_U(\mathcal{E}(U)),$ with respect to which the
$\mathcal{A}(U)$-morphism
\[(I+ f)_U:= I_U+ f_U
\]
is represented by the matrix
\[
\left(\begin{array}{ll} \mbox{I}_n & \mbox{H}\\ 0 &
\mbox{I}_n\end{array}\right)
\]
with ${}^tH= H.$
\end{enumerate}
\end{theorem}

\begin{proof}
$(\mathbf{1})$ From Lemma \ref{lemma2}, $\mbox{Im} f$ is totally
isotropic. Therefore, for any open subset $U$ of $X$ and sections
$s, t\in \mathcal{E}(U)$,
\[
\omega_U(f_U(s), f_U(t))=0.
\]
Since
\begin{eqnarray}
\omega_U((f^\ast)_Uf_U(s), t) & = & \omega_U((f_U)^\ast f_U(s), t)\nonumber \\
& = & \omega_U(f_U(s), f_U(t)) \nonumber \\
& = &  0 \nonumber
\end{eqnarray}
and $\omega$ is symplectic, it follows that
\[
(f^\ast)_Uf_U= (f^\ast_U)f_U=0.
\]
Thus,
\[
f^\ast f=0;
\]
since $f^\ast= -f,$ one reaches the desired property that $f^2=0.$

$(\mathbf{2})$ Fix an open set $U$ in $X$ and $s\in (\ker f)(U)=
\ker f_U,$ see \cite[p. 37, Definition 3.1]{tennison}. Moreover, let
$t\in \mathcal{E}(U);$ then
\[
\omega_U(s, f_U(t))=-\omega_U(f_U(s), t)=0.
\]
Thus,
\[
s\in (\mbox{Im}f)(U)^\perp\equiv f_U(\mathcal{E}(U))^\perp
\]
and hence
\[
(\ker f)(U)= \ker f_U\subseteq (\mbox{Im}f)(U)^\perp=
(\mbox{Im}f)^\perp(U)
\]
or
\[
\ker f\subseteq (\mbox{Im}f)^\perp\equiv f(\mathcal{E})^\perp.
\]

Conversely, let $t\in (\mbox{Im}f)^\perp(U)= (\mbox{Im}f)(U)^\perp.$
Then, for any $s\in (\mbox{Im}f)(U)\equiv \mbox{Im}f_U:=
f_U(\mathcal{E}(U))\equiv f(\mathcal{E})(U),$ one has
\[
\omega_U(t, s)=0.
\]
But $s= f_U(r) $ for some $r\in \mathcal{E}(U),$ therefore
\begin{equation}
\label{1.8} \omega_U(t, f_U(r))= -\omega_U(f_U(t), r)=0.
\end{equation}
Since (\ref{1.8}) is true for any $r\in \mathcal{E}(U),$
\[
f_U(t)=0,
\]
i.e.
\[
t\in (\ker f)(U):= \ker f_U.
\]
Hence,
\[
(\mbox{Im}f)^\perp(U)\subseteq (\ker f)(U)
\]
or
\[
(\mbox{Im}f)^\perp\subseteq \ker f.
\]

$(\mathbf{3})$ As $\mbox{Im} f\subseteq \ker f= (\mbox{Im}f)^\perp,$
so the sub-$\mathcal A$-module $\mbox{Im}f$ is totally isotropic.
Therefore, for any open $U\subseteq X$,
\[
\mbox{rank}(\mbox{Im}f)(U):= \mbox{rank}\ \mbox{Im}f_U\leq n.
\]
Now, let us fix an open set $U$ in $X$ and consider a basis $(s_1,
\ldots, s_k)$, $k\leq n,$ of $(\mbox{Im}f)(U)\equiv \mbox{Im}f_U.$
By \cite[Lemma 7]{malliosntumba2}, there exists a totally isotropic
sub-$\mathcal{A}(U)$-module $S$ of $\mathcal{E}(U)$, equipped with a
basis, which we denote
\[
(s_{k+1}, \ldots, s_{n+k})
\]
such that
\[
\begin{array}{ll}
\omega_U(s_i, s_{n+j})= \delta_{ij}, & \hspace{10mm}\mbox{for $i,
j=1, \ldots, k.$}
\end{array}
\]
Clearly,
\begin{equation}\label{1.9}
S\cap (\mbox{Im}f)^\perp(U)= S\cap (\ker f)(U)=0.
\end{equation}

As a result of (\ref{1.9}), the sum $S+ \mbox{Im}f_U$ is direct and
$S\oplus \mbox{Im}f_U$ is non-isotropic; therefore, by Theorem
\ref{theorem3}, one has
\[
\mathcal{E}(U)= (S\oplus \mbox{Im}f_U)\bot F
\]
for some sub-$\mathcal{A}(U)$-module $F$ of $\mathcal{E}(U)$, (cf.
\cite[Theorem 1]{malliosntumba2}). Since $F= (S\oplus
\mbox{Im}f_U)^\perp,$ $F$ is contained in $(\mbox{Im}f_U)^\perp=
(\mbox{Im}f)^\perp(U)= (\ker f)(U)$ and
\[
F^\perp= (\mbox{Im}f)(U):= \mbox{Im}f_U;
\]
i.e. $F$ is an orthogonal supplementary of $(\mbox{Im}f)(U)$ in
$(\ker f)(U).$ Since $F$ is free, non-isotropic and of rank $2n-2k$,
it can be equipped with a symplectic basis, say $(s_{k+1}, \ldots,
s_n, s_{n+k+1}, \ldots, s_{2n})$, see \cite{orthogonally}. As $s_1,
\ldots, s_n\in (\ker f)(U),$ it follows that
\[
(I_U+ f_U)(s_j)= s_j, \hspace{10mm} j= 1, \ldots, n.
\]
Therefore, if $H$ is the matrix representing $f_U$, $I_U+ f_U$ is
represented by the matrix
\[
\left(\begin{array}{ll} \mbox{I}_n & \mbox{H}\\ 0 &
\mbox{I}_n\end{array}\right),
\]
and this is a sympletic matrix if and only if ${}^tH= H$, ie. $H$ is
symmetric.
\end{proof}

\noindent PP Ntumba\\{Department of Mathematics and Applied
Mathematics}\\{University of Pretoria}\\ {Hatfield 0002, Republic of
South Africa}\\{Email: patrice.ntumba@up.ac.za}

\end{document}